\documentclass[11pt]{amsart}

\usepackage{graphics,graphicx}
\usepackage{tikz-cd}
\usetikzlibrary{calc}

\usepackage{dynkin-diagrams}

\usepackage{xcolor}

\usepackage{enumitem}

\usepackage{booktabs}

\usepackage[a4paper]{geometry}
\usepackage[utf8]{inputenc}
\usepackage[T1]{fontenc}

\usepackage{url}

\usepackage{amssymb,amsmath,amsthm}

\usepackage[backref=page,colorlinks=true,allcolors=blue]{hyperref} %

\newtheorem{thm}{Theorem}[section]

\newtheorem{lem}[thm]{Lemma}
\newtheorem{prop}[thm]{Proposition}

\newtheorem*{ack}{Acknowledgements}

\theoremstyle{definition}
\newtheorem{defn}[thm]{Definition}
\newtheorem{rem}[thm]{Remark}
\newtheorem{exa}[thm]{Example}

\DeclareMathOperator{\SL}{SL}
\DeclareMathOperator{\SO}{SO}
\DeclareMathOperator{\PSO}{PSO}

\DeclareMathOperator{\GL}{GL}
\DeclareMathOperator{\Sp}{Sp}

\DeclareMathOperator{\PGL}{PGL}

\DeclareMathOperator{\diag}{diag}
\DeclareMathOperator{\Ric}{Ric}
\DeclareMathOperator{\Bl}{Bl}
\DeclareMathOperator{\Vol}{Vol}

\newcommand{\bbC}{\mathbb{C}}

\newcommand{\bbZ}{\mathbb{Z}}
\newcommand{\bbP}{\mathbb{P}}
\newcommand{\bbR}{\mathbb{R}}

\begin{document}

\title{Limits of conical Kähler-Einstein metrics on rank one horosymmetric spaces}

\author{Thibaut Delcroix}
\address{Thibaut Delcroix, Univ Montpellier, CNRS, Montpellier, France}
\email{thibaut.delcroix@umontpellier.fr}
\urladdr{http://delcroix.perso.math.cnrs.fr/}

\date{\today}

\begin{abstract}
We consider families of conical Kähler-Einstein metrics on rank one horosymmetric Fano manifolds, with decreasing cone angles along a codimension one orbit. 
At the limit angle, which is positive, we show that the metrics, restricted to the complement of that orbit, converge to (the pull-back of) the Kähler-Einstein metric on the basis of the horosymmetric homogeneous space, which is a projective homogeneous space. 
Then we show that, on the symmetric space fibers, the rescaled metrics converge to Stenzel's Ricci flat Kähler metric.  
\end{abstract}

\maketitle


\section{Introduction}

Chi Li and Song Sun showed in \cite{Li_Sun_2014} that there exists conical Kähler-Einstein metrics on \(\bbP^2\), with conic singularities along a quadric \(D\subset \bbP^2\), for cone angles \(2\pi\beta\), with \(\beta\in ]1/4,1]\). 
Equivalently, there exists a (singular) Kähler-Einstein metric on the pair \((\bbP^2,(1-\beta)D)\) for \(\beta\in ]1/4,1]\). 
Then they conjectured that these metrics converge in Gromov-Hausdorff sense to the weighted projective space \(\bbP(1,1,4)\), equipped with its orbifold Kähler-Einstein metric. 
This convergence is now well understood thanks to the machinery developped by Chen, Donaldson and Sun to solve the Yau-Tian-Donaldson \cite{CDS2}, but before that Chi Li used numerical analysis to support the conjecture in \cite{Li_2015}, and further observed numerically that the bubble out of this convergence was the \(\bbZ/2\bbZ\)-quotient of the Eguchi-Hanson Ricci flat Kähler metric. 
Our purpose in this paper is to prove exactly (instead of only numerically) the bubbling phenomenon in this case, and in a generalization to limits of conical Kähler-Einstein metrics on rank one horosymmetric varieties. 
The bubbles arising in our generalization are fibrewise Stenzel's Ricci flat Kähler metrics on symmetric spaces of rank one. 

The setting is the following. 
Let \(G\) be a connected complex semisimple group. 
We consider a \(G\)-homogeneous fibre bundle \(X\) on a projective homogeneous space \(G/P\), with fibers \(P\)-equivariantly isomorphic to the unique equivariant compactification \(X_0\) of a rank one semisimple group \(G_0/H_0\). 
In that situation, \(X\) possesses a unique prime \(G\)-stable divisor \(D\), which is the unique closed orbit of \(G\). 
Assume that \(X\) is Kähler-Einstein, and let \([0,b[\subsetneq [0,1[\) be the set of all \(s\) such that \((X,sD)\) admits a singular log Kähler-Einstein metric in \(c_1(X)\). 
We denote by \(\omega_s\) a log Kähler-Einstein metric on \((X,sD)\) which is invariant under a maximal compact subgroup \(K\) of \(G\). Such metrics have conical singularities along \(D\) by \cite{Guenancia_Paun_2016}. 

\begin{thm}
As \(s\) converges to \(b\), the conical Kähler-Einstein metrics \(\omega_s\) restricted to \(X\setminus D\) converge to the pullback of the Kähler-Einstein metric on \(G/P\). 
Furthermore, restricted to a symmetric fiber \(G_0/H_0\), and rescaled, they converge to Stenzel's complete Ricci flat Kähler metric. 
\end{thm}

We note that the divisor \(D\) is a projective homogeneous space, and as such, is a Fano Kähler-Einstein manifold. 
However in the general case considered here, it is not assumed to be an ample divisor in \(X\), or a multiple of the anticanonical divisor. 
The simplest example where the fibration structure is non-trivial is given by \(\bbP^2\times \bbP^2\) equipped with the diagonal \(\SL_3\)-action, see Example~\ref{P^2^2}.

The proof relies on the translation of the equation governing the existence of \(K\)-invariant log Kähler-Einstein metrics to a single ODE. 
This is obtained through the convex geometric translation of Kähler geometry of horosymmetric varieties obtained by the author in \cite{DelHoro}.  
One then follows the strategy of \cite{DelKE}, inspired by \cite{Wang_Zhu_2004}, and already specialized to the rank one case by Biquard and the author in \cite{KRFSS}. 
The added input is to examine precisely what fails in the existence result as \(s\to b\), and translate this into estimates on the sequence of convex potentials. 

As should be obvious from the proof, and comparison with \cite{HoroCMab} for example, the same methods should apply with very limited changes, to more general settings. 
For example, one could use the same method to show that Stenzel's metric appear in bubbling phenomenon for various continuity paths or flows, such as the Kähler-Ricci flow, or the twisted Kähler-Einstein continuity path for the existence of Kähler-Einstein metrics.
The methods apply as well on a singular rank one horosymmetric variety, and in that setting we exhibited examples without Kähler-Einstein metrics in \cite{KRFSS}. 

Several generalizations of this work are of interest, let us comment on some of these to conclude the introduction. 
First, rank one horosymmetric varieties do not exhaust the possibilities for cohomogeneity one manifolds. However, the Kähler geometry of such manifolds in general is not yet fully understood. We intend to fill this gap one day, and we expect Stenzel's metrics to stealth into this setting as well. 

If one wishes to extend the picture to manifolds without symmetries, a natural replacement to Stenzel's metrics is provided by Tian-Yau's asymptotically conical complete Ricci flat Kähler metrics on the complement of a divisor \cite{Tian_Yau_1991}. 
Biquard and Guenancia informed the author that they studied this problem in general by gluing techniques (with a strategy analogous to \cite{Biquard_Guenancia_2022}), obtaining a wide generalization. 

Finally, it should be noted that Tian-Yau's metrics are constructed on the complement of a smooth divisor, representing a multiple of the anticanonical line bundle. 
This is of course not the only setting of interest, but the case of singular divisor for example, is wide open in general. 
The construction by Biquard and the author of asymptotically conical Ricci flat Kähler metrics on rank two symmetric spaces, with singular tangent cones in \cite{KRFSS}, should be revisited with the point of view of limits of conical Kähler-Einstein metrics, and may provide further insights on possible generalizations. 

\begin{ack}
The author is partially funded by ANR-21-CE40-0011 JCJC project MARGE and ANR-18-CE40-0003 JCJC project FIBALGA.
We thank Olivier Biquard, Henri Guenancia and Chenyang Xu for discussions related to the topic of this note.  
\end{ack}

\section{Rank one horosymmetric manifolds, and their Kähler geometry}

We extract from the general setting of horosymmetric varieties in \cite{DelHoro, HoroCMab} the Kähler geometry tools to translate our problem into the study of a single ODE. 

\subsection{Rank one horosymmetric manifolds}

Let \(G\) be a connected, semisimple complex algebraic group. 
We always denote by the corresponding \emph{fraktur} character the Lie algebra of a group, for example, \(\mathfrak{g}\) is the Lie algebra of \(G\). 

\begin{defn}
An algebraic subgroup \(H\) of \(G\) is \emph{horosymmetric} if there exists a parabolic subgroup \(P\) of \(G\), a Levi decomposition \(P=LP^u\) of \(P\) and a complex Lie algebra involution \(\sigma\) of \(L\) such that 
\[\mathfrak{h} = \mathfrak{l}^{\sigma} \oplus \mathfrak{p}^u. \]
It is of \emph{rank one} if furthermore the reductive symmetric space \(L/L^{\sigma}\) is a rank one symmetric space. 
\end{defn}

From now on we assume that \(G\), \(P=LP^u\), \(H\) are fixed, with the assumption that the action of \(L\) on \(L/L\cap H\) does not factor through a torus. 
If \(H\) is a horosymmetric subgroup of rank one, we also say that \(G/H\) is a rank one horosymmetric homogeneous space.

One should be aware that we do not require the action of \(L\) to be effective, nor to have a finite kernel, otherwise the class of rank one horosymmetric spaces would be reduced to rank one symmetric spaces. 
Thanks to the classification of rank one symmetric spaces, the last sentence of the definition means that \(L\) admits a semisimple quotient \(G_0\) such that the action of \(L\) on \(L/L\cap H\) factors through \(G_0\), with isotropy group \(H_0\), and up to isogeny, \(G_0\) and \(H_0)\) are in the following list: 
\begin{itemize}
\item \(G_0 = \SO_{n+2}\) and \(H_0=\SO_{n+1}\), for some \(n\geq 0\)  
\item \(G_0 = \SL_{n+1}\) and \(H_0=\GL_n\), for some \(n\geq 2\) 
\item \(G_0 = \Sp_{2n}\) and \(H_0=\Sp_2\times \Sp_{2n-2}\), for some \(n\geq 3\)
\item \(G_0=F_4\) and \(H_0=\SO_9\). 
\end{itemize}

The first case with \(n=0\) is special: in this case \(G_0=\SO_2\simeq \bbC^*\) is a torus, and the homogeneous space si called horospherical (of rank one). 
We will not consider this case in the present paper since in that situation, conical Kähler-Einstein metrics exist only for one given cone angle, so it does not make sense to take a limit. 

\begin{rem}
We can describe \(G/H\) as the fiber bundle with fiber \(G_0/H_0\) associated to the \(P\)-principal fiber bundle \(G\to G/P\). This is the quotient of \(G\times G_0/H_0\) by the action of \(P\) given by \(p\cdot(g,x)=(gp^{-1},p\cdot x)\) where the action of \(P\) on \(G_0/H_0\) is via the morphisms \(P/P^u\simeq L \to G_0\).  
\end{rem}

\begin{exa}
\label{P^2^2}
Let \(G=\SL_3\) and consider the natural diagonal action on \(\bbP^2\times \bbP^2\). 
There are two orbits: the diagonal and its complement. 
The complement \(\bbP^2\times \bbP^2\setminus \diag(\bbP^2)\) is a rank one horosymmetric space. 
Indeed, consider the stabilizer \(H\) of the point \([1:0:0],[0:1:0])\). 
It is the subgroup generated by the maximal torus \(T\) of diagonal matrices, and the unipotent radical of the parabolic subgroup \(P\) of \(G\), which is the stabilizer of the hyperplane \(\{z=0\}\) in \(\bbP^2\) with homogeneous coordinates \([x:y:z]\). 
We can choose the Levi subgroup \(L=S(\GL_2\times \GL_1)\) of block diagonal matrices in \(P\). 
Then \(L\cap H=T\), the action of \(L\) on \(L\cap H\) factors through \(\PGL_2\) and the rank one symmetric space fiber is \(\PGL_2/\PSO_2\).
\end{exa}

\subsection{Maximal compact group action}

As in \cite{DelHoro, HoroCMab}, we make the following choices to reduce Kähler geometry of \(G/H\) to convex geometry and combinatorics. 

Choose \(T_s\) a torus in \(L\), maximal for the property that \(\sigma\) acts on \(T_s\) as the inversion. 
Choose \(T\) a maximal \(\sigma\)-stable torus of \(L\) containing \(T_s\). 
Let \(Q\) denote the parabolic subgroup of \(G\) which is opposite to \(P\), that is, \(Q\cap P=L\). 
Let \(B\) be a Borel subgroup of \(G\) with \(T\subset B\subset Q\) such that, if \(\beta\) is a root of \(B\cap L\) (with respect to the maximal torus \(T\)), then either \(\sigma(\beta)=\beta\) or \(-\sigma(\beta)\) is a root of \(B\cap L\).  
We denote by \(\Phi^+\) the roots of \(B\), by \(\Phi_{Q^u}\) the roots of \(Q^u\), and by \(\Phi_s^+\) the roots of \(B\cap L\) which are not fixed by \(\sigma\). 

Fix \(K\) a maximal compact subgroup. 
Let \(\mathfrak{a}_s\) denote the real vector space \(\mathfrak{t}_s\cap i\mathfrak{k}\). 

Under the \emph{rank one} assumption, the torus \(T_s\) is of complex dimension one, and \(\mathfrak{a}_s\) is a one dimensional real vector space. 
The restriction of any root of \(\Phi_s^+\) to \(\mathfrak{a}_s\) is non-zero, and they all define the same non-negative closed half-space \(\mathfrak{a}_s^+\subset \mathfrak{a}_s\), called the \emph{positive restricted Weyl chamber}. 

\begin{prop}
The image of the positive restricted Weyl chamber \(\mathfrak{a}_s^+\) under the map \(\mathfrak{g}\to G/H, x\mapsto \exp(x)H\) is a fundamental domain for the action of \(K\) on \(G/H\). 
\end{prop}

\subsection{Invariant metrics as functions}

On the homogeneous space \(G/H\), we will work with metrics in (a multiple of) the anticanonical class \(c_1(K_{G/H}^{-1})\). 
It should be noted that in general, the space of \(K\)-invariant Kähler classes up to scaling may be of positive dimension. 

We let \(\chi_H\) be the character of \(H\) defined by 
\[ \chi_H = \left.\left(\sum_{\alpha\in \Phi_s^+\cup \Phi_{Q^u}} \alpha\right)\right|_{H} \]
We will consider it as the \(\sigma\)-invariant element of \(\mathfrak{t}^*\) defined by 
\[ \chi_H = \sum_{\alpha\in \Phi_s^+\cup \Phi_{Q^u}} \frac{\alpha+\alpha\circ\sigma}{2} \]
This is the isotropy character of \(K_{G/H}^{-1}\), that is, the character corresponding to the one-dimensional representation of \(H\) defined by the fiber at the identity coset. 
We let also \(\pi\) denote the quotient map \(G\to G/H\). 

Let \(\omega\) be a \(K\)-invariant closed real (1,1)-form in \(c_1(K_{G/H}^{-1})\). 
Recall that it is the curvature of a \(K\)-invariant hermitian metric \(h\) on \(K_{G/H}^{-1}\). 
Note that \(h\) is well-defined only up to a multiplicative constant. 

Let \(\gamma\) be the primitive generator of \(X^*(T/T\cap H)\) which evaluates non-negatively on elements of \(\mathfrak{a}_s^+\). 
Let \(\gamma^*\) be the unique element of \(\mathfrak{a}_s^+\) such that \(\gamma(\gamma^*)=1\). 

We choose \(\xi\) a non-zero element of the fiber of \(K_{G/H}^{-1}\) at the identity coset. 
Define \(u_{h}:\bbR\to \bbR\) by 
\[ u_{h}(x)= -\ln \lvert \exp(x\gamma)\cdot \xi\rvert_h. \]
This is an even function which fully determines \(h\), hence \(\omega\). 
Even though it is only determined by \(\omega\) up to an additive constant, we will denote by \(u_{\omega}\) the function associated to an arbitrary choice of \(h\) and \(\xi\). 

The form \(\omega\) is Kähler if and only if \(u_{\omega}''>0\) and \(u'(x) < \min_{\alpha\in \Phi_{Q^u}} \frac{\kappa(\chi_H,\alpha)}{\kappa(\gamma,\alpha)}\).

\begin{prop}
Assume that \(\omega\) is Kähler, then up to an additive constant, \(u_{\Ric(\omega)}\) is defined by, for \(x>0\), 
\begin{equation} 
\label{Ricci_toric_potential}
u_{\Ric(\omega)}(x) = -\frac{1}{2}\ln \left( u''(x) \prod_{\beta\in \Phi_s^+}\frac{\kappa(\beta,u'(x)\gamma)}{\sinh(2\beta(x\gamma^*))}\prod_{\alpha\in \Phi_{Q^u}}\kappa(\alpha,\chi_H-u'(x)\gamma) \right) 
\end{equation}
\end{prop}

From a direct application of \cite{DelHoro}, in the product in formula~\eqref{Ricci_toric_potential}, there should be the factor 
\[ \prod_{\alpha\in \Phi_{Q^u}} e^{2\alpha(x\gamma^*)} \]
but in our situation, this is equal to \(1\) for all \(x\in \bbR\). 
Indeed, the action of \(K\) induces, on \(\mathfrak{a}_s\), the action of the restricted Weyl group, which in our rank one case is the reflection with respect to the origin. 
By \cite{Rossmann}, this reflection on \(\mathfrak{a}_s\) is induced by the action on \(\mathfrak{t}\) of an element \(w\) of the Weyl group of \(L\) with respect to \(T\). 
Any element of this Weyl group induces a permutation of the roots in \(\Phi_{Q^u}\). 
As a consequence, \(\sum_{\alpha\in \Phi_{Q^u}} \alpha\) is invariant under \(w\), and thus for any \(x\in \bbR\), 
\(\sum_{\alpha\in \Phi_{Q^u}} \alpha (x\gamma^*) = \sum_{\alpha\in \Phi_{Q^u}} \alpha (-x\gamma^*) = 0\).

\subsection{On the toroidal equivariant compactification}

Any rank one symmetric space \(G_0/H_0\) admits a unique \(G_0\)-equivariant compactification \(X_0\), which is smooth (it is either a projective space or a quadric).
In this paper, for a rank one horosymmetric space \(G/H=(G\times G_0/H_0)/P\), we consider the \(G\)-equivariant compactification \(X\) given by the fiber bundle \((G\times X_0)/P\) with fiber \(X_0\), called the toroidal compactification.
By definition, it is smooth as well.  

\begin{exa}
For \(G/H=\SL_3/\langle T, P^u\rangle = \bbP^2\times \bbP^2\setminus \diag(\bbP^2)\), the toroidal compactification is \(X=\Bl_{\diag(\bbP^2)}(\bbP^2\times \bbP^2)\). 
\end{exa}

In general there can be more than one \(G\)-equivariant compactification, as is obvious in the above example: \(\bbP^2\times\bbP^2\) is also an equivariant compactification. 
However, for all compactifications, there is only one added orbit, and it is a divisor only in the toroidal compactification. 

Let \(D=X\setminus (G/H)\) denote the closed \(G\)-orbit in \(X\), which is the unique \(G\)-stable prime divisor on \(X\).  

Let \(\lambda_{ac} = 1 + \sum_{\alpha\in \Phi_s^+\cup \Phi_{Q^u}}\alpha(\gamma^*) = 1 + \sum_{\alpha\in \Phi_s^+}\alpha(\gamma^*)\). 

\begin{prop}
A smooth \(K\)-invariant Kähler form \(\omega\) in \(c_1(K_{G/H}^{-1})\) extends to \(X\) as a locally bounded Kähler current in \(c_1(X)-s[D]\) if and only if 
\( u'_{\omega}(\bbR_+) = [0,\lambda_{ac}-s[ \).  
\end{prop}

\subsection{EDO and their translation}

In view of the previous description of \(K\)-invariant Kähler metrics by one-variable real functions, the existence of canonical Kähler metrics on rank one horosymmetric varieties can be encoded by ODEs. 
We give explicitly the equation for the metrics that will be of interest for us. 
To shorten the notations, we introduce 
\begin{equation} 
v(p) = \prod_{\beta\in \Phi_s^+}\kappa(\beta, p\gamma) \prod_{\alpha\in \Phi_{Q^u}}\kappa(\alpha, \chi_H-p\gamma) 
\end{equation}
and 
\begin{equation}
J(x) = \prod_{\beta\in \Phi_s^+} \sinh(2\beta(\gamma^*)x), \qquad j=-\ln J
\end{equation} 

\begin{prop}
\label{prop_sing_KE}
A \(K\)-invariant singular Kähler-Einstein metric exists on the pair \((X,sD)\) if and only if there exists a smooth, even function \(u:\bbR\to \bbR\) such that \(u''(\bbR)\in ]0,+\infty[\), \(u'(\bbR_+) = [0,\lambda_{ac}-s[\), and for all \(x>0\), 
\begin{equation} 
\label{eqn_sing_KE}
u'' v(u') = e^{-(2u+j)} 
\end{equation}
Furthermore, this function is uniquely determined. 
Finally, there exists such a function if and only if 
\begin{equation} 
\label{criterion}
\int_0^{\lambda_{ac}-s}pv(p)\mathop{dp} > (1-\lambda_{ac})\int_0^{\lambda_{ac}-s}v(p)\mathop{dp} 
\end{equation}
\end{prop}

We refer to \cite{BBEGZ} for the full definition and caracterization of singular Kähler-Einstein metrics, and to \cite{KRFSS} for the proof of criterion~\eqref{criterion}.
We will go over a part of the proof from \cite{KRFSS} in the next section. 

\begin{prop}
Assume that \(G/H=G_0/H_0\), that is, \(G/H\) is a rank one symmetric space. 
Then the \(K\)-invariant Kähler metric \(\omega_u\) in the trivial Kähler class \(c_1(K_{G/H}^{-1})\) is a Ricci flat Stenzel metric if and only if 
\begin{equation} 
u'' v(u') = CJ 
\end{equation}
for some \(C\in \bbR\).  
\end{prop}

\section{Estimates break}

By Proposition~\ref{prop_sing_KE}, there exists a singular Kähler-Einstein metric on \(G/H\) as long as criterion~\eqref{criterion} is satisfied. 
Obviously, the set of \(s\) satisfying criterion~\eqref{criterion} is bounded above by \(1\). 
We assume that it is non-empty and denote its supremum (which is not a maximum) by \(b\). 
Note that it is always non-empty in the case when \(G/H=G_0/H_0\) is a rank one symmetric space, since in that case \(X_0\) is a projective homogeneous space, hence admits a Kähler-Einstein metric for \(s=0\). 
We will also impose throughout that \(s\geq 0\). 
Although a lower bound is needed in the proof, allowing negative values of \(s\) may allow to deal with more examples, when \(X_0\) does not admit a Kähler-Einstein metric. Since we do not know of the existence of such an example, we settle with the lower bound \(s\geq 0\). 

For \(0\leq s <b\), we denote by \(u_s\) the solution, and let 
\[ \nu_s := 2u_s+j \]
In view of the expression of \(j\), we know that \(\nu_s\) is smooth, strictly convex on \(]0,+\infty[\), with \(\nu_s'(]0,+\infty[) = ]-\infty, 1-s[\). 
In particular, since \(1-s>0\),  \(\nu_s\) admits a unique minimum. 
We define \(m_s\) and \(x_s\) by 
\[ m_s=\min_{]0,+\infty[}\nu_s = \nu_s(x_s) \]
We define also \(y_s\) and \(\delta_s\) by 
\[
    [y_s-\delta_s,y_s+\delta_s]:=\nu_s^{-1}([m_s,m_s+1])
\]

\begin{lem}
There exist a constant \(\varepsilon_2>0\) independent of \(0\leq s <b\) such that \(y_s-\delta_s \geq \varepsilon_2\). 
\end{lem}

\begin{proof}
Let \(0<x<1\) and consider the difference 
\[ \nu_s(x)-\nu_s(1) = 2(u_s(x)-u_s(1))+j(x) -j(1) \]
Since \(u_s\) is \(\lambda_{ac}\)-Lipschitz (even \(\lambda_{ac}-s\)-Lipschitz),  and \(\nu_s(1)\geq m_s\), we have 
\[ \nu_s(x) - m_s \geq -\lambda_{ac} + j(x) - j(1) \]
The right hand side is a function of \(x\) which is independent of \(s\) and  we have 
\[ \lim_{x\to 0} j(x) = \lim_{x\to 0} -\sum_{\beta\in \Phi^+}\ln \sinh (2\beta(\gamma^*)x) \to +\infty \]
hence there exists a \(\varepsilon_2>0\) such that the right hand side is larger than \(1\) for \(x\geq \varepsilon_2\).  
\end{proof}

\begin{lem}
There exists \(\tilde{\delta}>0\) independent of \(0\leq s < b\) such that 
\[ [x_s-\tilde{\delta},x_s+\tilde{\delta}] \subset [y_s-\delta_s,y_s+\delta_s] \]
\end{lem}

\begin{proof}
On \([\varepsilon_2,+\infty]\), the derivative of \(\nu_t\) is uniformly bounded: 
\[ j'(\varepsilon_2)\leq 2u_s'(\varepsilon_2) + j'(\varepsilon_2) = \nu_s'(\varepsilon_2) \leq \nu_s'(x) = 2u_s'(x)+j'(x) \leq 2(\lambda_{ac}-s)-\sum_{\beta\in \Phi_s^+}2\beta(\gamma^*) \leq 2(1-s) \]
We thus have a uniform Lipschitz bound on \(\nu_s\). 
Since \(\nu_s(x_s)=m_s\) and \([y_s-\delta_s,y_s+\delta_s]=\nu_s^{-1}([m_s,m_s+1])\), this yields the result. 
\end{proof}

\begin{lem}
\label{lem_intermediate}
There exists a constant \(C>0\) such that for \(0\leq s <b\), 
\[ \delta_s \leq C e^{\frac{m_s}{2}} \]
\end{lem}

\begin{proof}
By assumptions on the solution \(u_s\), we know that \(v(u_s'(x))\) is positive for \(x\in ]0,+\infty[\) for any \(s\in [0,b[\). Let 
\[ C_0:=\inf_{]0,\lambda_{ac}[} \frac{1}{v} \]

For \(s\in [0,b[\) fixed, consider the function \(f\) defined by  
\[ f(x) = \nu_s(x)-C_0e^{-m_s-1}((x-y_s)^2-\delta_s^2) - m_s -1 \]
By direct computation, we have 
\begin{align*} 
f''(x) & = 2u_t''(x) + j''(x) - 2C_0e^{-m_s-1} \\ 
& = 2u_t''(x) - 2C_0e^{-m_s-1} \\ 
 & \geq 2\frac{e^{-\nu_s(x)}}{v(u_s'(x))} - 2C_0e^{-m_s-1} \\
 & \geq 2C_0\left( e^{-\nu_s(x)}- e^{-m_s-1} \right) 
\end{align*}
As a consequence, \(f\) is convex on \([y_s-\delta_s,y_s+\delta_s]\). 
Since \(f(y_s-\delta_s)=f(y_s+\delta_s)=0\), by convexity we deduce \(f(y_s)\leq 0\) and thus 
\[ \nu_s(y_s) \leq -C_0e^{-m_s-1}\delta_s^2 + m_s +1 \]
Since \(\nu_s(y_s)\geq m_s\) by definition of \(m_s\), we gather 
\[ \delta_s^2 \leq \frac{1}{C_0e^{-m_s-1}} \]
hence the result.
\end{proof}

\begin{lem}
\label{lem_volume}
We have, for any \(s\in [0,b[\), 
\[ 0< V_-:=\int_0^{\lambda_{ac}-b} v(y) \mathop{dy} \leq \int_0^{\infty} e^{-\nu_s} \mathop{dx} \leq V_+:=\int_0^{\lambda_{ac}} v(y) \mathop{dy}  \]
\end{lem}

\begin{proof}
By equation~\eqref{eqn_sing_KE}, and the change of variable \(y=u_s'(x)\), we have 
\[
    \int_0^{\infty} e^{-\nu_s} \mathop{dx} = \int_0^{\infty} v(u_s')u_s'' \mathop{dx} = \int_0^{\lambda_{ac}-s} v(y) \mathop{dy}  
\]
hence the result.
\end{proof}

\begin{lem}
There exist a constant \(C_m\) independent of \(s\in [0,b[\) such that 
\[ \lvert m_s \rvert \leq C_m \]
\end{lem}

\begin{proof}
Write 
\begin{equation}
\label{coarea}
    \int_0^{+\infty} e^{-\nu_s} = e^{-m_s}\int_0^{+\infty} e^{-\lambda} \Vol(\nu_s^{-1}([m_s,m_s+\lambda])) \mathop{d\lambda} 
\end{equation}
We will obtain both an upper bound and a lower bound, and compare the two. 

By convexity, \(\nu_s^{-1}([m_s,m_s+\lambda])\) is included in the \(\lambda\)-dilation with center \(x_s\) of \([y_s-\delta_s, y_s+\delta_s]\), hence 
\[ \Vol(\nu_s^{-1}([m_s,m_s+\lambda]) \leq 2\delta_t \lambda \]
Integrating yields  
\begin{align*} 
\int_0^{+\infty} e^{-\nu_s} & \leq 2\delta_s e^{-m_s} \int_0^{+\infty}\lambda e^{-\lambda} \mathop{d\lambda}  \\
& \leq 2\delta_s e^{-m_s} \\ 
& \leq 2Ce^{-\frac{m_s}{2}} 
\end{align*}
by Lemma~\ref{lem_intermediate}. 
By Lemma~\ref{lem_volume}, we thus have 
\[ 0< V_- \leq 2Ce^{-\frac{m_s}{2}}  \]
in other words, 
\[ m_s \leq 2\ln(2C/V_-)  \]

On the other hand, for \(\lambda\geq 1\), we have \([x_s-\tilde{\delta},x_s+\tilde{\delta}]\subset \nu_s^{-1}([m_s,m_s+1])\subset \nu_s^{-1}([m_s,m_s+\lambda])\), hence 
\[ \Vol(\nu_s^{-1}([m_s,m_s+\lambda]) \geq 2 \tilde{\delta} \]
Now putting this in formula~\eqref{coarea} yields 
\begin{align*}
V_+ \geq \int_0^{+\infty} e^{-\nu_s} & \geq 2\tilde{\delta}e^{-m_s}\int_{1}^{+\infty}e^{-\lambda}\mathop{d\lambda} \\
& \geq 2\tilde{\delta}e^{-m_s-1}
\end{align*}
In other words, 
\[ m_s \geq \ln(2\tilde{\delta}/V_+)-1 \]
\end{proof}

\begin{lem}
\label{linear_growth}
There are constants \(C_{\delta}\), \(l_1\), \(l_0\) independent of \(s\in[0,b[\) such that \(\delta_s\leq C_{\delta}\) and for all \(0<x<+\infty\), 
\[ \nu_t(x) \geq l_1\lvert x-x_s \rvert +l_0 \]
\end{lem} 

\begin{proof}
The bound for \(\delta_s\) follows from the two previous lemma. 
As a consequence, we have \(\nu_s(x_s\pm 2C_{\delta})\geq m_s+1\). 
By convexity, we deduce that \(\nu_s(x)\geq \frac{\lvert x-x_s \rvert}{2C_{\delta}} +m_s\) for \(x\) not in the interval \([x_s-2C_{\delta},x_s+2C_{\delta}]\). 
Since \(\nu_s\geq m_s\) everywhere and \(\frac{\lvert x-x_s \rvert}{2C_{\delta}}\leq 1\) on the latter interval, we obtain the result, with \( l_0 := m_s-1 \) and \(l_1 := \frac{1}{2C_{\delta}}\). 
\end{proof}

Now we draw the first consequence of the non-existence of solutions at \(s=b\). 

\begin{lem}
\[ \lim_{s\to b} x_s =+\infty\] 
\end{lem}

\begin{proof}
Assume the contrary, then there exists a subsequence of \(0\leq s_k < b\) such that \(s_k\to b\) and \(x_s\to x_{b} \in ]\varepsilon_2,+\infty[\). 
In this case, by the same arguments as in \cite{KRFSS}, the functions \(u_{s_k}\) converge locally uniformly to a solution of equation~\eqref{eqn_sing_KE}. 
This is impossible. 
\end{proof}

\section{Collapsing}

Let \(r=\operatorname{Card}(\Phi_s^+)\) be  the order of vanishing of \(v\) at \(0\). 
Note that \(r+1\) is equal to the dimension of the symmetric space fiber \(G_0/H_0\). 

\begin{lem}
\label{inverted_derivative}
There exists a strictly increasing function \(V^{-1}:[0,+\infty[\to [0,+\infty[\) and a constant \(C_{-1}>0\), such that for any \(q\in [0,+\infty[\),
\[ \frac{1}{C_{-1}}q^{\frac{1}{r+1}} \leq V^{-1}(q) \leq C_{-1} q^{\frac{1}{r+1}} \]
and for any \(s\in [0,b[\), for any \(x\in [0,+\infty[\), 
\[ u'_s(x) = V^{-1}\left( \int_0^{x}e^{-\nu_s(y)} \mathop{dy} \right)  \]
\end{lem}

\begin{proof}
By definition of \(v\) and \(\lambda_{ac}\), we know that \(v\) is continuous on \([0,\lambda_{ac}]\), positive on \(]0,\lambda_{ac}[\). 
By definition of \(r\), we deduce that there exists a constant \(C_v>0\) such that for any \(p\in [0,\lambda_{ac}-1]\),  
\begin{equation}
\label{v_near_zero}
\frac{1}{C_v} p^r \leq v(p) \leq C_v p^r 
\end{equation}
Let \(V:[0,\lambda_{ac}]\to \bbR\) be the primitive of \(v\) which vanishes at \(0\). For \(p\in [0,\lambda_{ac}-1]\), we have
\[ \frac{1}{(r+1)C_v} p^{r+1} \leq V(p) \leq \frac{C_v}{r+1} p^{r+1} \]
Since \(v\) is positive on \(]0,\lambda_{ac}[\), \(V\) is strictly increasing on \([0,\lambda_{ac}]\), and its inverse \(V^{-1}:V([0,\lambda_{ac}])\to [0,\lambda_{ac}]\) is a strictly increasing function. 

Let \(p_0\in ]0,\lambda_{ac}-1[\) be such that \(\frac{C_v}{r+1} p^{r+1}\in V([0,\lambda_{ac}-1])\) for any \(p\in [0,p_0]\). 
Then applying \(V^{-1}\) to the above inequality yields 
\[ V^{-1}\left(\frac{1}{(r+1)C_v} p^{r+1}\right) \leq p \leq V^{-1}\left(\frac{C_v}{r+1} p^{r+1}\right)  \]
Setting \(q_0:=\frac{C_v}{r+1} p_0^{r+1}\) we have, for \(0\leq q\leq q_0\), 
\[ \left(\frac{r+1}{C_v}\right)^{\frac{1}{r+1}} q^{\frac{1}{r+1}} \leq V^{-1}(q) \leq \left(C_v(r+1)\right)^{\frac{1}{r+1}} q^{\frac{1}{r+1}} \]
Since \(V^{-1}\) is increasing and \(V^{-1}(q_0)>0\), we deduce that there exists a constant \(C_{-1}>0\) such that for all \(q\in V([0,\lambda_{ac}])\), 
\[ \frac{1}{C_{-1}}q^{\frac{1}{r+1}} \leq V^{-1}(q) \leq C_{-1} q^{\frac{1}{r+1}} \]

We set \(V^{-1}(q):=C_{-1} q^{\frac{1}{r+1}}\) for \(q\in \bbR_+\setminus V([0,\lambda_{ac}])\). 
This is an increasing function, albeit no longer continuous, which still satisfies the above inequality. 

Integrating equation~\eqref{eqn_sing_KE} between \(x_1<x_2 \in [0,+\infty[\) yields  
\[ V(u'_s(x_2)) - V(u'_s(x_1)) = \int_{x_1}^{x_2} e^{-\nu_s(y)} \mathop{dy} \]
In particular, for \(x_1=0\), 
\[ u'_s(x_2) = V^{-1}\left( \int_0^{x_2}e^{-\nu_s(y)} \mathop{dy} \right) \]
\end{proof}

\begin{lem}
\label{slow}
There exists a constant \(C_u>0\) such that for any \(s\in [0,b[\) and any \(x\in [0,x_s]\), 
\[ 0 \leq u_s(x)-u_s(0) \leq C_{-1} \left(\frac{e^{-l_0}}{l_1}\right)^{\frac{1}{r+1}}\frac{r+1}{l_1}\left( e^{\frac{l_1}{r+1}(x-x_s)}-e^{-\frac{l_1}{r+1}x_s} \right) \]
\end{lem}

\begin{proof}
For \(0\leq x \leq x_s\), we can plug the uniform linear growth inequality from Lemma~\ref{linear_growth} in the expression of \(u_s'\) from Lemma~\ref{inverted_derivative} to get 
\begin{align*}
u'_s(x) & = V^{-1}\left( \int_0^{x}e^{-\nu_s(y)} \mathop{dy} \right) \\
&  \leq V^{-1}\left( \int_0^{x}e^{-l_1(x_s-y)-l_0} \mathop{dy} \right) \\
&  \leq V^{-1}\left( \frac{e^{-l_1x_s-l_0}}{l_1} (e^{l_1x}-1) \right) \\
& \leq V^{-1}\left( \frac{e^{-l_0}}{l_1} e^{l_1(x-x_s)} \right) \\
& \leq C_{-1} \left(\frac{e^{-l_0}}{l_1}\right)^{\frac{1}{r+1}} e^{\frac{l_1}{r+1}(x-x_s)} \\
\end{align*}
As a consequence, by integrating we get, for \(0\leq x \leq x_s\),  
\begin{align*}
0 \leq u_s(x)-u_s(0)  & = \int_0^x u_t'(y)\mathop{dy} \\
& \leq \int_0^x C_{-1} \left(\frac{e^{-l_0}}{l_1}\right)^{\frac{1}{r+1}} e^{\frac{l_1}{r+1}(y-x_s)} \mathop{dy} \\
& \leq C_{-1} \left(\frac{e^{-l_0}}{l_1}\right)^{\frac{1}{r+1}}\frac{r+1}{l_1}\left( e^{\frac{l_1}{r+1}(x-x_s)}-e^{-\frac{l_1}{r+1}x_s} \right) 
\end{align*}
\end{proof}

The collapsing result follows immediately~: for any \(M>0\), the functions \(u_s|_{[0,M]}-u_s(0)\) converge uniformly to the zero function, hence the metrics associated to \(u_s\) converge to the metric associated with the zero function, which is exactly the pullback of the Kähler-Einstein on the basis \(G/P\). 

\section{Convergence to Stenzel's metric on (rescaled) symmetric fibers}

Our main theorem follows from the more precise result. 

\begin{thm}
Consider the functions 
\[ \underline{u}_s = e^{\frac{2u_s(0)}{r+1}}(u_s - u_s(0)) \]
Then up to subsequence of \(s\), they converge locally uniformly to an even, \(C^2\) solution to the equation 
\[ v_0(u')u'' = C J \]
for some constant \(C>0\), 
with \(v_0(p)=\prod_{\alpha\in \Phi_s^+} \kappa(\alpha,p\gamma)\), 
hence they converge to the potential of a Stenzel metric on the symmetric fiber. 
\end{thm}

\begin{proof}
Let \(M>0\). 
Since \(x_s\to \infty\), we may as well assume that \([0,M]\subset [0,x_s/2]\) for all (sufficiently close to \(b\)) \(s\). 
We assume for the first part of the proof that \(x\in [0,M]\). 

We first derive uniform estimates for \(\underline{u}_s'\). 
By Lemma~\ref{inverted_derivative}, and since \(u_s(0)\leq u_s(x)\leq u_s(x_s/2)\), we have 
\[ V^{-1}\left( e^{-2u_s(x_s/2)}\int_0^x J \right) \leq u_s'(x) = V^{-1}\left( \int_0^x e^{-2u_s}J \right) \leq V^{-1}\left( e^{-2u_s(0)}\int_0^x J \right) \]
hence 
\[ \frac{1}{C_{-1}}e^{-\frac{2u_s(x_s/2)}{r+1}}\left(\int_0^x J\right)^{\frac{1}{r+1}} \leq u_s'(x) \leq C_{-1}e^{-\frac{2u_s(0)}{r+1}}\left(\int_0^x J\right)^{\frac{1}{r+1}} \]
Finally, 
\[ \frac{1}{C_{-1}}e^{-\frac{2(u_s(x_s/2)-u_s(0))}{r+1}}\left(\int_0^x J\right)^{\frac{1}{r+1}} \leq \underline{u}_s'(x)=e^{\frac{2u_s(0)}{r+1}}u'_s(x) \leq C_{-1}\left(\int_0^x J\right)^{\frac{1}{r+1}} \]
Since \(u_s(x_s/2)-u_s(0)\) converges uniformly to \(0\) as \(s\to b\), and \(\int_0^x J\) is independent of \(s\), we obtain a uniform bound for \(\underline{u}_s'\) on \([0,M]\). 

By integration from \(0\), since \(\underline{u}_s(0)=0\), we obtain a uniform bound for \(\underline{u}_s\) on \([0,M]\) as well. 

Turning to \(\underline{u}_s''\), we have, at least for \(x\neq 0\),
\[ \underline{u}_s''(x) = e^{\frac{2u_s(0)}{r+1}}u_s''(x) = e^{\frac{2u_s(0)}{r+1}}\frac{e^{-2u_s(x)}J(x)}{v(u_s'(x))} \]
By inequality~\eqref{v_near_zero}, we have 
\[ e^{\frac{2u_s(0)}{r+1}}\frac{e^{-2u_s(x)}J(x)}{C_v(u_s'(x))^r} \leq 
\underline{u}_s''(x) \leq 
e^{\frac{2u_s(0)}{r+1}}\frac{C_v e^{-2u_s(x)}J(x)}{(u_s'(x))^r} \]
Plugging in the inequality for \(u_s'(x)\) obtained above, we get 
\[ \frac{1}{C_vC_{-1}^r}e^{2(u_s(0)-u_s(x))}J(x)\left(\int_0^xJ\right)^{-\frac{r}{r+1}} \leq 
\underline{u}_s''(x) \leq 
C_vC_{-1}^r e^{2(\frac{u_s(0)}{r+1}+\frac{ru_s(x_s/2)}{r+1}-u_s(x))}J(x)\left(\int_0^xJ\right)^{-\frac{r}{r+1}}
\]
Again, this provides a uniform bound on \([0,M]\), since \(u_s(x)\) converges uniformly to \(u_s(0)\) on \([0,x_s/2]\), and \(J(x)\left(\int_0^xJ\right)^{-\frac{r}{r+1}}\) is a function on \(]0,+\infty[\), independent of \(s\), which extends continuously to \(0\) in view of the order of vanishing of \(J\) at \(0\). 

It follows from Arzela-Ascoli theorem that up to subsequence the \(\underline{u}_s\) converge locally uniformly in \(C^1\) sense to a limit function \(\underline{u}_b\) on \([0,+\infty]\). 
From the equation, we deduce that the convergence is actually in \(C^2\) sense. 
It remains to check the equation satisfied at the limit. 

For this, we consider first 
\[ v(u_s'(x)) = \prod_{\beta\in \Phi_s^+}\kappa(\beta, u_s'(x)\gamma) \prod_{\alpha\in \Phi_{Q^u}}\kappa(\alpha, \chi_H-u_s'(x)\gamma) \]
We have seen that \(u_s'\) converges uniformly on \([0,x_s/2]\) to \(0\) as \(s\to b\). 
As a consequence, the factors corresponding to \(\alpha\in \Phi_{Q^u}\) converge to constants. 
In view of the relations \(u_s'=e^{-\frac{2u_s(0)}{r+1}}\underline{u}_s'\) and \(u_s''=e^{-\frac{2u_s(0)}{r+1}}\underline{u}_s''\), passing to the limit in Equation~\eqref{eqn_sing_KE} multiplied by \(e^{2u_s(0)}\) yields 
\[ \underline{u}_s''\prod_{\beta\in \Phi_s^+}\kappa(\beta, \underline{u}_s'(x)\gamma) \prod_{\alpha\in \Phi_{Q^u}}\kappa(\alpha, \chi_H) = J \]
hence the result. 
\end{proof}

\bibliographystyle{alpha}
\bibliography{Stenzel_bubbles}

\end{document}